\documentclass[12pt]{article}
\usepackage{amsmath,amsfonts,amssymb,amsthm}

\newtheorem{proposition}{Proposition}[section]
\newtheorem{theorem}[proposition]{Theorem}
\newtheorem{corollary}[proposition]{Corollary}

\newtheorem{definition}[proposition]{Definition}
\newtheorem{example}[proposition]{Example}

\newcommand{\nc}{\newcommand}
\nc{\I}{{\bf 1}}
\nc{\bG}{{G}}
\nc{\bT}{\mathbf{T}}
\nc{\bA}{\mathbf{A}}
\nc{\bN}{{N}}
\nc{\bM}{{M}}
\nc{\cB}{{\mathcal B}}
\nc{\cE}{{\mathcal E}}
\nc{\cF}{{\mathcal F}}
\nc{\cG}{{\mathcal G}}
\nc{\cS}{{\mathcal S}}
\nc{\cM}{{\mathcal M}}
\nc{\R}{{\mathbb R}}
\nc{\N}{{\mathbb N}}
\nc{\Z}{{\mathbb Z}}

\nc{\BP}{\mathbb{P}}
\nc{\BE}{\mathbb{E}}
\nc{\BQ}{\mathbb{Q}}
\numberwithin{equation}{section}

\begin{document}
\renewcommand{\thefootnote}{\fnsymbol{footnote}}
\author{{\sc G\"unter Last\footnote{Institut f\"ur Stochastik, 
Karlsruhe Institute of Technology, 76128 Karlsruhe, Germany} 
and Hermann Thorisson\footnote{Science Institute, University of Iceland, 
Dunhaga 3, 107 Reykjavik, Iceland}}}
\title{What is typical\,?}
\date{\today}
\maketitle

\begin{abstract} 
\noindent 
Let $\xi$ be a random measure on a locally compact second 
countable topological group and let $X$ be a random element in 
a measurable space on which the group acts. In the compact 
case, we give a natural definition of the concept that the origin 
is a typical location for $X$ in the mass of  $\xi$, and prove that 
when this holds the same is true on sets placed uniformly at random 
around the origin. This new result motivates an extension of the concept of typicality
to the locally compact case where it coincides with the concept of mass-stationarity. 
We describe recent developments in Palm theory where these 
ideas play a central~role.
\end{abstract}

\vspace{0.2cm}

\vspace{0.2cm}

\section{Introduction}

The word \,`typical'\, is sometimes used in probability contexts in an
informal way. For instance, a typical element in a finite set, --
or in a finite interval, -- is usually 
interpreted as an element chosen according to the uniform distribution.
Also, after adding a point at the origin to a stationary Poisson process,
the new point is often referred to as a typical point of the process.
Note that in both these examples the choice of an element (point) 
is far from being arbitrary;
typical does not mean arbitrary.
In this paper we attempt to make the term `typical' precise.

We consider a random measure $\xi$ on a locally compact second 
countable topological group and a random element $X$ in 
a measurable space on which the group acts. In the compact 
case, we give a natural definition of the concept that the origin 
is a typical location for $X$ in the mass of  $\xi$, and prove that 
this property is equivalent to the more mysterious property that 
the same is true on sets placed uniformly at random 
around the origin. 
This new result motivates an extension of the concept of typicality 
to the locally compact case where it coincides 
with the concept of mass-stationarity
which was introduced in \cite{LaTho09}.
We then outline recent developments  in Palm theory 
of stationary random measures where these concepts
play a central role.

\section{Preliminaries}\label{preliminaries}

Let $G$ be a locally compact second countable topological group
equipped with the Borel $\sigma$-algebra $\cG$.
Then the mapping from $G \times G$ to $G$ taking $(s, t)$ to  
$st$ and the mapping from $G$ to $G$ taking
$t$ to  $t^{-1}$ are measurable. 
We refer to the
neutral element $e$ of $G$ as the {\em origin}
and to the elements of $G$ as {\em locations}.
 
 For a measure $\mu$ on $(G, \cG)$ and a set
$C \in \cG$ such that $0 < \mu(C) < \infty$, define 
the conditional probability measure
$\mu(\, \cdot \mid C)$ by 
\begin{align*}
\mu(A \mid C) = \mu(A \cap C)/\mu(C), \quad A \in \cG.
\end{align*} 
For convenience, we let
\,$\mu(\, \cdot \mid C)$\, equal some fixed probability
measure if \,$\mu(C)=0$. 
For $t\in G$, let  
$t\mu$ be
the pushforward
of $\mu$ under the mapping $s\mapsto ts$,
that is,
\begin{align*}
t \mu (A) := \mu(t^{-1}A), \quad A \in \cG.
\end{align*}
Let $\lambda\ne 0$ be a left-invariant Haar
measure, see e.g.\ Theorem 2.27 in [6].
An example is any countable group 
$G$ with $\lambda$ the counting measure.
Another example is $\R^d$ under addition
with $\lambda$ the Lebesgue measure.

Let  $\overset{D}{=}$ denote identity in distribution.\! 
Let $\xi$ be a  nontrivial 
random measure on~$(G, \cG)$.
Say that $\xi$ is {\em stationary} if
\begin{align*}
t \xi \overset{D}{=} \xi, \quad t \in G.
\end{align*}
Let $G$ act on a measurable space $(E, \cE)$ measurably, that is,
such that the mapping
from $G \times E$ to $G$ taking $(t, x)$ to  
$t x$ is measurable. 
Let  $X$ be a random element in $(E, \cE)$.
For~instance, $X$ could be a random field $X = (X_s)_{s \in G}$ and 
$t X = (X_{t^{-1} s})_{s \in G}$ for $t \in G$.
Say that $X$ is {\em stationary} if
\begin{align}\label{(1.1)}
t X \overset{D}{=} X,\quad t \in G.
\end{align}
Put $t (X,\xi) = (t X,t \xi)$.
Say that $(X,\xi)$ is {\em stationary} if
\begin{align*}
t (X,\xi) \overset{D}{=} (X,\xi), \quad t \in G.
\end{align*}
Let $(\Omega, \cF, \BP)$ be the probability space on which
the random elements in this paper are defined.
If $S$  is a random element in $(G, \cG)$, let $S^{-1}$ denote the
group inverse of $S$ (and not the inverse of $S$ as a function
defined on $\Omega$).

\section{Compact groups and typicality}\label{finite}

In this section assume that $G$ is compact. 
Then both $\lambda$ and $\xi$ are finite and $\lambda$ is also right invariant 
(see e.g.\ Theorem 2.27  in \cite{Kallenberg}).
An example is any finite group with $\lambda$ the counting measure. 
Another example is the $d$-dimensional rotation group.

Let $S$  be a random element in $(G, \cG)$. 
Say that $S$ is uniformly distributed on $C \in \cG$ if 
$S$ has the distribution $\lambda(\, \cdot \mid C)$. 
Note that $\lambda(\, \cdot \mid G)= \lambda /\lambda(G)$.

\begin{definition}\label{2.1}\rm
$(a)$
If $S$ is uniformly distributed on $G$, 
then $S$ is a {\em typical} location in~$G$.

$(b)$
If $S$ is a typical location in $G$ and independent of $X$, 
then $S$ is a typical location {\em for}  $X$.

$(c)$
If $S$ is a typical location  for $X$ and $S^{-1}X \overset{D}{=} 
X$, then the {\em origin} is a typical location for $X$.
\end{definition}

\begin{theorem}\label{2.2}
Let $G$ be compact.

$(a)$ If $S$ is a typical location for $X$,
then $S^{-1} X$ is stationary.

$(b)$
The origin is a typical location for $X$
if and only  if $X$ is stationary.
\end{theorem}
\begin{proof}
$(a)$ 
If $S$ is a typical location for $X$
then so is $St^{-1} $ for each $t \in G$.
Thus $(St^{-1})^{-1}X$ 
has the same distribution as $S^{-1}X$.
But $(St^{-1})^{-1}X = t (S^{-1}X) $.
Thus $S^{-1} X$ is stationary.

$(b)$ Let $S$ be a typical location for $X$. If 
$S^{-1}X \overset{D}{=} X$
then $X$ is stationary since $S^{-1}X$ is stationary.
Conversely, if $X$ is stationary then 
$S^{-1}X \overset{D}{=} X$ follows from
\eqref{(1.1)} and the independence of $S$ and $X$. 
\end{proof}

We shall now extend the above typicality concepts from the uniform 
distribution to random measures. 

\begin{definition}\label{2.3}\rm
$(a)$
If the conditional
distribution of $S$ given $\xi$ is $\xi(\cdot \mid G)$, 
then  $S$ is a {\em typical} location in the {\em mass} of $\xi$.

$(b)$
If $S$ is a typical location in the mass of $\xi$ and 
$S^{-1}\xi\overset{D}{=}\xi$, 
then  
the {\em origin} is a  typical location in the~mass~of~$\xi$.

$(c)$
If $S$ is a typical location in the mass of $\xi$ and
conditionally independent of $X$~given~$\xi$, 
then $S$ is a typical location {\em for}  $X$
in the mass of $\xi$.

$(d)$
If $S$ is a typical location for $X$ in the mass of $\xi$ and 
$S^{-1}(X, \xi)\overset{D}{=}(X,\xi)$, then  
the {\em origin} is a typical location for $X$ in the mass of $\xi$.
\end{definition}

The following theorem says that
 the origin is a typical location for $X$ in the mass of 
$\xi$ if and only if
it is a typical location for $X$ in the mass of $\xi$
on {\em sets placed uniformly at random 
around the origin}.

\begin{theorem}\label{2.4}
Let $G$ be compact. 
 Then the origin is a typical location for $X$ in the mass of $\xi$ if and only if
for all  $C \in \cG$ such that $\lambda(C) > 0$ 
\begin{gather}\label{(2.1)}
\left(V_C^{-1}(X, \xi), U_CV_C\right) \overset{D}{=} \left((X, \xi), U_C\right)
\end{gather}
where 
\begin{align*}
&\text{$(i)$ \,\,$U_C$ is uniformly distributed on $C$ and independent of $(X, \xi)$}, and\\
&\text{$(ii)$ $V_C$ has the conditional distribution $\xi(\,\cdot \mid U_C^{-1}C)$ given $(X, \xi, U_C)$}.
\qquad \quad
\end{align*}
\end{theorem}
\begin{proof}
Suppose \eqref{(2.1)} holds for all $C$. Then in particular
$V_G^{-1}(X, \xi) \overset{D}{=} (X, \xi)$. 
Moreover, since $U^{-1}_GG=G$ we have from $(ii)$ that
$V_G$ has the conditional distribution $\xi\left(\cdot \mid G\right)$ given $(X, \xi, U_G)$.
This implies that $V_G$ is a typical location in the mass
of $\xi$ and also that 
$V_G$ is conditionally independent of $X$ given $\xi$.
Thus the origin is a
typical location for $X$ in the mass of $\xi$.

Conversely, suppose  the origin is a typical location for $X$ 
in the mass of $\xi$. 
For nonnegative measurable $f$ and with $U_C$ and $V_C$ as above
we have
\begin{align*}
\BE\big{[}f\big{(}V_C^{-1}(X, \xi), U_CV_C\big{)}\big{]} \!=
\BE\bigg{[}\iint  \!1_{\{u\in C\}}1_{\{v\in u^{-1}C\}}f\left(v^{-1}(X, \xi)\right), uv)
\frac{\xi(dv)}{\xi(u^{-1}C)}\frac{\lambda(du)}{\lambda{(C)}}\bigg{]}.
\end{align*}
Let $S$ be a typical location for $X$ in the mass of $\xi$. Then we obtain 
\begin{align*}
\BE&\big{[}f(V_C^{-1}(X, \xi), U_CV_C)\big{]}\\ 
&=
\BE\bigg{[}\iint  1_{\{u\in C\}}1_{\{v\in u^{-1}C\}}f\big{(}v^{-1}S^{-1}(X, \xi), uv\big{)}
\frac{(S^{-1}\xi)(dv)}{(S^{-1}\xi)(u^{-1}C)}
\frac{\lambda(du)}{\lambda{(C)}}\bigg{]}\\ 
&=
\BE\bigg{[}\iint  1_{\{u\in C\}}1_{\{S^{-1}v\in u^{-1}C\}}
f\big{(}(S^{-1}v)^{-1}S^{-1}(X, \xi), uS^{-1}v\big{)}
\frac{\xi(dv)}{(S^{-1}\xi)(u^{-1}C)}\frac{\lambda(du)}{\lambda{(C)}}\bigg{]}\\ 
&=
\BE\bigg{[}\iiint  1_{\{u\in C\}}1_{\{s^{-1}v\in u^{-1}C\}}
f\big{(}v^{-1}(X, \xi), us^{-1}v\big{)}
\frac{\xi(dv)}{(s^{-1}\xi)(u^{-1}C)}\frac{\lambda(du)}{\lambda{(C)}}
\frac{\xi(ds)}{\xi(G)}\bigg{]}.
\end{align*}
Make the variable  substitution 
$r = us^{-1}v$ (equivalently, $u= rv^{-1}s$) and use right-invariance
of $\lambda$ to obtain 
\begin{align*}
\BE&[f(V_C^{-1}(X, \xi), U_CV_C)]\\ 
&=\BE\bigg{[}\iiint1_{\{v^{-1}s\in r^{-1}C\}}1_{\{r\in C\}}
f\big{(}v^{-1}(X, \xi), r\big{)}\frac{\xi(dv)}{(v^{-1}\xi)(r^{-1}C)}
\frac{\lambda(dr)}{\lambda{(C)}}
\frac{\xi(ds)}{\xi(G)}\bigg{]}\\
&=\BE\bigg{[}\iiint1_{\{s\in r^{-1}C\}}1_{\{r\in C\}}
f\big{(}v^{-1}(X, \xi), r\big{)}\frac{\xi(dv)}{v^{-1}\xi(r^{-1}C)}
\frac{\lambda(dr)}{\lambda{(C)}}
\frac{v^{-1}\xi(ds)}{\xi(G)}\bigg{]}\\
&=\BE\bigg{[}\iint1_{\{s\in r^{-1}C\}}1_{\{r\in C\}}
f\big{(}S^{-1}(X, \xi), r\big{)}\frac{(S^{-1}\xi)(ds)}{(S^{-1}\xi)(r^{-1}C)}
\frac{\lambda(dr)}{\lambda{(C)}}
\bigg{]}.
\end{align*}
Again, apply the fact that $S$ is a typical location for $X$ in the mass of $\xi$  
(and recall we are assuming that  the origin is a typical location for $X$ 
in the mass of $\xi$)
to obtain
\begin{align*}
\BE&[f(V_C^{-1}(X, \xi), U_CV_C)]
=\BE\bigg{[}\iint 1_{\{s\in r^{-1}C\}}1_{\{r\in C\}}
f\big{(}(X, \xi), r\big{)}\frac{\xi(ds)}{\xi(r^{-1}C)}\frac{\lambda(dr)}{\lambda{(C)}}
\bigg{]}\\
&=\BE\bigg{[}\int \bigg{(}\int 1_{\{s\in r^{-1}C\}}\frac{\xi(ds)}{\xi(r^{-1}C)}
\bigg{)}1_{\{r\in C\}}
f\big{(}(X, \xi), r\big{)}\frac{\lambda(dr)}{\lambda{(C)}}
\bigg{]}\\
&=\BE\bigg{[}\int1_{\{r\in C\}}f\big{(}(X, \xi), r\big{)}
\frac{\lambda(dr)}{\lambda{(C)}}\bigg{]}\\&=\BE\big{[}f\big{(}(X, \xi\big{)}, U_C)\big{]}
\end{align*}
that is, \eqref{2.1} holds.
In the above calculation expressions like
$\left((s^{-1}\xi)(u^{-1}C)\right)^{-1}$ can be given some
fixed (arbitrary) value if $(s^{-1}\xi)(u^{-1}C)=0$. This requires 
some care but can be accomplished as in the first
part of the proof of Theorem 6.3 in [8].
\end{proof}

\section{Locally compact groups, typicality and\\ mass-stationarity}\label{infinite}

We shall now drop the condition that $G$ is compact.
Then $\lambda$ and $\xi$ are only $\sigma$-finite
so Definitions~\ref{2.1} and \ref{2.3} do not work. 
 However, Theorem~\ref{2.4}
suggests a way to define 
typicality of the origin in this case: 
demand that the origin  
is a typical location for $X$ in the mass of $\xi$
on sets placed uniformly at random 
around the origin.

\begin{definition}\label{3.1}\rm
$(a)$
If \eqref{(2.1)}  holds for all relatively compact $\lambda$-continuity 
sets $C$ with $\lambda(C) > 0$, 
then  
the {\em origin} is a {\em typical} location {\em for} $X$ in the {\em mass} of~$\xi$.

$(b)$ 
If \eqref{(2.1)}  holds with $X$ deleted, 
then we say that 
the origin is a typical location in the mass of $\xi$. 

$(c)$
If $(a)$ is true with $\xi = \lambda$, 
then we say that 
the origin is a typical location for $X$.
\end{definition}

The reason we choose here to restrict $C$ to be a
$\lambda$-continuity set (that is, a set with boundary having 
$\lambda$-measure zero) is that then the property in
the definition is exactly the property used in \cite{LaTho09} 
to define {\em mass-stationarity}\,: 
$(X, \xi)$ is called mass-stationary if 
the origin is a typical location for $X$
in the mass of $\xi$ in the sense of Definition~\ref{3.1}.

Now recall (see e.g.\ [6] for the case $G=\R^d$
and [7] for the general case)
that a pair $(X, \xi)$ is called a {\em Palm version} 
of a stationary pair
$(Y, \eta)$ if for all nonnegative measurable functions $f$ and all
compact $A \in \cG$ with $\lambda (A) > 0$, 
\begin{align} \label{(3.1)}
\BE[f(X, \xi)] = \BE\Big{[}\int_A f\big{(}t^{-1} (Y, \eta)\big{)} \eta(dt)\Big{]}
\Big{/}\lambda(A).
\end{align}
In this definition  \,$(X,\xi)$\, and \,$(Y, \eta)$\, are allowed to
have distributions that are only $\sigma$-finite 
and not necessarily probability measures.
The distribution of $(X,\xi)$ is finite 
if~and~only if $\eta$ has finite intensity, that is, 
if and only if $\BE[\eta(A)]<\infty$ for compact $A$.
In this case the distribution of $(X,\xi)$ can be normalized to a probability 
measure.

The following equivalence of mass-stationarity
and Palm versions was established~in \cite{LaTho09} in the Abelian case and
extended to the non-Abelian case in \cite{La08b}.

\begin{theorem}\label{3.2}
Let $G$ be locally compact
and allow the distributions of $(X, \xi)$ and $(Y, \eta)$ to be only
$\sigma$-finite.
Then $(X, \xi)$ is mass-stationary
$($that is, the origin is a typical location for $X$ in the mass of $\xi$$)$ 
if and only if $(X,\xi)$ is the  Palm version of a stationary $(Y, \eta)$.
\end{theorem}
An important ingredient in the proof of this theorem is the intrinsic
characterization of Palm measures derived in \cite{Mecke}.

\section{The Poisson process and reversible shifts} \label{poisson} 

We now turn to the other example 
mentioned in the introduction. 
This example concerns a stationary Poisson process $\eta$
to which we add a point at the origin, thereby yielding the 
process $\xi := \eta + \delta_0$.
In this setting, the new point is often referred to as a typical point of $\xi$.

For the Poisson process on the line ($G = \R$), this is motivated 
by the fact that  
the intervals between the points of $\xi$
have i.i.d.\ (exponential) lengths and thus if
the origin is shifted to the $n^{\text{th}}$ point on the right (or on the left)
then the distribution of the process does not change:
\begin{align}\label{(4.1)}
\xi(T_n + \cdot) \overset{D}{=} \xi, \quad n \in \Z,
\end{align}
where $T_0 := \pi_0(\xi) := 0$ and
\begin{align*}
T_n  := \pi_n(\xi) := \begin{cases} 
\text{ $n^{\text{th}}$ point on the right of the origin if  $n > 0$, }\\
\text{ $-n^{\text{th}}$ point on the left of the origin if  $n < 0$.}
\end{cases} 
\end{align*}
Since $\xi$ looks distributionally the same from all its points, it is 
natural to say that 
the point at the origin is a typical point of $\xi$.

It is well known that on the line the typicality property \eqref{(4.1)} 
characterizes 
Palm versions $\xi$ of 
stationary simple point processes $\eta$ 
(but it is only in the Poisson case that the Palm
version is of the form $\eta + \delta_0$).
Thus due to Theorem~\ref{3.2}, \eqref{(4.1)} 
is equivalent to the origin being a typical location in the
mass of $\xi$ in the sense of Definition~\ref{3.1}.
Thus , -- on the line, -- calling the point at the origin a typical point
is not only natural because of \eqref{(4.1)} but also
consistent with Definition~\ref{3.1}. 

The property \eqref{(4.1)} is a more transparent 
definition of typicality 
than  Definition~\ref{3.1}, but it does not extend immediately  
beyond the line: if $d > 1$ and we go out from the origin in any fixed 
direction then we will (a.s.)\ not hit a point of the Poisson process. 
One might conceive of mending this by ordering the points 
according to their distance from the origin, but this does not yield \eqref{(4.1)}
as is clear from the following example.

\begin{example}\label{4.1}\rm
If $\xi = \eta + \delta_0$ 
is the Palm version of a Poisson process $\eta$
and we shift the origin to the point $T$ that is closest
to the origin,
then the Poisson property is lost:  
the shifted process $\xi(T + \cdot)$
is sure to have a point (the point at the old origin $-T$) that is
closer to the point at the origin 
than to any other point of $\xi(T + \cdot)$.
This is not a property of $\xi$ as the following argument shows.

The stationary Poisson process $\eta$ 
need not have a point that is
closer to the origin
than to any  other point of $\eta$ since 
there is a positive probability that  $\eta$ has no point
in the unit ball around the origin 
and that a bounded shell around that ball
is covered by the balls
of diameter $\frac{1}{2}$
with centers at the points in the shell.
\end{example}

Thus for the Poisson process in the plane ($G = \R^2$), --
and in higher dimensions ($G = \R^d$) and beyond, --
there is no obvious motivation 
(save the analogy with the line)
for calling the new point at the origin typical. 
However, adding that point to the stationary Poisson process
yields its Palm version, and by Theorem~\ref{3.2}
the origin is a typical location in the mass of the Palm version.
Thus calling the point at the origin a typical point
is again consistent with Definition~\ref{3.1}. 

Now although the property \eqref{(4.1)}  does not extend immediately  
beyond the line, a generalization
of \eqref{(4.1)} does.
The key property of $\pi_n$ defining $T_n$ in 
 \eqref{(4.1)} is that they are {\em reversible}\,: 
 a measurable map $\pi$ taking each
$\xi$ having a point at the origin
to a point $T = \pi (\xi)$ is reversible if it has a {\em reverse} $\pi'$
 such that 
 $$
 \text{$\pi'(\xi(T + \cdot)) = -T$ \quad and \quad 
  $\pi(\xi(T' + \cdot)) = -T'$\,\,where\,\, $T' = \pi' (\xi)$.}
  $$
 Above, the shift from the point at the origin to
 the $n^\text{th}$ point on the right (or left) is reversed by
 shifting back to the $n^\text{th}$ point on the left (or right).
In Example~\ref{4.1} on the other hand, the shift to the closest point 
is not reversible because there can be more than one point having
a particular point as their closest points. 
The following example of reversible $\pi_n$ yielding a
generalization of \eqref{(4.1)} is from \cite{FLT}.

\begin{example}\label{4.2}\rm
Let $d = 2$ and consider $\xi = \eta + \delta_0$ where $\eta$ is a 
stationary Poisson process in $\R^2$. Link the points of $\xi$ into a tree 
by defining the mother of each point as follows: 
place an interval of length one around the point parallel to the $x$-axis
and send the interval off in the direction of the $y$-axis until it hits a point,
let that point be the mother of the point we started from. 
Define the age-order of sisters by the order of their $x$ coordinates.
This procedure (see \cite{FLT}) links the points
into a one-ended tree such that each point has an ancestor 
with a younger sister.

Now  put 
\begin{align*}
\pi(\xi) = 
\begin{cases}
\text{oldest daughter of $0$, if $0$ has a daughter,}\\
\text{oldest younger sister, if $0$ has a younger sister  but no daughter, }\\
\text{oldest younger sister of youngest ancestor who has a younger sister, else.}
\end{cases}
\end{align*}
This $\pi$ is reversible with reverse $\pi'$ defined by
\begin{align*}
\pi'(\xi) = 
\begin{cases}
\text{mother of $0$, if $0$ has no older sister,}\\
\text{youngest older sister, if $0$ has a daughterless youngest older sister, }\\
\text{last in youngest-daughter offspring-line  of the youngest older sister, else.}
\end{cases}
\end{align*}
 Put $T_0 := \pi_0(\xi) := 0$ and recursively for $n > 0$
 \begin{align*}
 T_n := \pi_n(\xi) &:=  \pi(\xi(T_{n-1} + \cdot))\\
  T_{-n}  := \pi_{-n}(\xi) &:= \pi'(\xi(T_{-(n-1)} + \cdot)).
 \end{align*}
With this enumeration of the points of $\xi$
the typicality property \eqref{(4.1)} holds, see \cite{FLT}. 

For $d > 2$ the same approach works  to establish \eqref{(4.1)}.
 In that case place a 
$d - 1$ dimensional unit ball around each point and send
the ball off in the $d^\text{th}$ dimension until it hits a point. 
When $d = 3$, this again
strings up all the points of $\xi$ 
into the integer line. However when $d > 3$, this yields an infinite forest of
trees, and the tree containing the point at the origin
only strings up a subset of the points, see \cite{FLT}. 
\end{example}

More sophisticated tree constructions can be found in \cite{HP03}
and \cite{Timar04}.
In particular,  the~points can be linked
into a single tree in all dimensions. And this is true not only
for the Poisson process but for Palm versions of
arbitrary stationary aperiodic simple point processes in $\R^d$.

\section{Simple point  processes and point-stationarity}
\label{point-stationarity}

The 
property \eqref{(4.1)} is a well known characterization of
 Palm versions $\xi$ of stationary
simple point processes on the line. When a random element $X$ is involved
and $(X, \xi)$ is the Palm version of a stationary pair then the 
characterization reads as follows 
(recall that $T_n^{-1}=-T_n$
is the group inverse of $T_n$):
\begin{align*}
T_n^{-1}(X, \xi) \overset{D}{=} (X, \xi), \quad n \in \Z.
\end{align*}
This  is  implied by  
the following property,
\begin{align}\label{(5.1)}
T^{-1} (X, \xi) \overset{D}{=} (X, \xi) 
\text{ for all $T = \pi(\xi)$ where $\pi$ is reversible,} 
\end{align}
which is in turn 
implied by the following property,
\begin{align}\label{(5.2)}
T^{-1} (X, \xi) \overset{D}{=} (X, \xi) 
\text{ for all $T = \pi(X,\xi)$ where $\pi$ is reversible;} 
\end{align}
here $\pi$ {\em reversible} means that 
$\pi$ has a {\em reverse} 
$\pi'$ such that  $\pi'(T^{-1}(X, \xi)) = T^{-1}$ and 
$\pi(T'^{-1}(X, \xi)) = T'^{-1}$ where $T' = \pi' (X, \xi)$.
 
The latter two properties are 
 not restricted to the line, as we saw in Example~\ref{4.2}. 
In \cite{He:La:05}  and \cite{He:La:07} the property 
\eqref{(5.1)} 
is used to define {\em point-stationarity},
a precursor  of mass-stationarity. 
There it is proved, for simple point processes on Abelian $G$,
$(i)$ that  point-stationarity  characterizes Palm versions of stationary pairs,
$(ii)$ that \eqref{(5.1)} can be replaced by \eqref{(5.2)}, and $(iii)$ that
in \eqref{(5.1)} it suffices to consider $\pi$ 
such that  $\pi' = \pi$ (such $\pi$ are said to  induce a 
{\em matching}).

Point-stationarity was
introduced earlier in \cite{Thor99} (see also \cite{Thor00}) for simple point
processes on $G = \R^d$, but the
definition there was more cumbersome, involving
{\em stationary independent backgrounds}\,: 
a random element $Z$ (possibly
defined on an extension of the
underlying probability space) is a 
stationary independent background for $(X,\xi)$ if 

\vspace{0.3 cm}
$(i)$\,\, $Z$ takes values in a measurable space on which $G$ acts measurably, and

\vspace{0.1 cm}
$(ii)$ $Z$ is stationary and independent of $(X, \xi)$.
 
\vspace{0.3 cm} 
\noindent 
In \cite{Thor99}\, $\xi$ is a  simple point process on $\R^d$ and
the pair $(X, \xi)$
is called point-stationary if for all stationary independent backgrounds $Z$,
\begin{align}\label{(5.3)}
T^{-1} ((Z,X), \xi) \overset{D}{=}  ((Z,X), \xi) 
\text{ for all $T = \pi((Z,X),\xi)$ where $\pi$ is reversible.} 
\end{align}
This property was proved 
to characterize Palm versions $(X, \xi)$ of stationary
pairs and to be equivalent to what later became
the definition of mass-stationarity. The proof of the fact that \eqref{(5.3)} 
implies \eqref{(2.1)} 
with $C  = [0,1)^{d}$ is sketched in the following 
example. The result for $C  = [0,h)^{d}$ is obtained in the same way,
and the result for
relatively compact $C$ then follows by a simple conditioning argument.

\begin{example}\label{5.1}\rm Consider $G = \R^d$. 
Let  $U_C$ be uniform on $C = [0,1)^{d}$ 
and $U$ be uniform on $[0, 1)$. Let $U_C$ and $U$ be independent 
and independent of $(X,\xi)$. Put $Z = (U_C^{-1}\Z, U)$
and let shifts leave $U$ intact.
Let $\pi_n(Z, \xi)$ be the
$n^{\text{th}}$ point of $\xi$ 
after the point at the origin 
in the circular lexicographic ordering of the points 
in the set $U_C^{-1}C$. 
These $\pi_n$ are reversible (with $\pi'_n$ obtained from the reversal
of the lexicographic ordering), and so is the mapping $\pi$ defined by
$$
\pi((Z,X), \xi) := \pi(Z,\xi) := \pi_{[U\xi(U_C^{-1}C)]}(Z, \xi).
$$ 
Now $V_C :=\pi(Z,\xi)$ has the conditional distribution 
$\xi(\,\cdot\! \mid U_C^{-1}C)$ given $((Z,X), \xi)$,
and thus also given $(X, \xi, U_C)$ since $U_C$ and $Z$ are measurable
functions of each other. Thus 
\eqref{(5.3)} implies \eqref{(2.1)} for this particular set $C$. 
\end{example}

The results mentioned above together with Theorem~\ref{3.1} 
yield the following theorem.

\begin{theorem} \label{5.2}
Let $\xi$ be a simple point process on a locally compact Abelian $G$
having a point at the origin.
Allow the distributions of $(X, \xi)$ and $(Y, \eta)$ to be only
$\sigma$-finite.
Then the following claims are equivalent:

$(a)$ 
the pair $(X, \xi)$ is mass-stationary,

$(b)$ 
the pair $(X, \xi)$ is the Palm version of a stationary $(Y, \eta)$,

$(c)$ 
the pair $(X, \xi)$ is point-stationary,

$(d)$
the property \eqref{(5.1)} holds with $\pi$ restricted to be its own 
reverse (matching),

$(e)$
the property \eqref{(5.2)} holds,

$(f)$
the property \eqref{(5.3)} holds for all stationary independent backgrounds $Z$.
\end{theorem}
\begin{proof}
The only claim that has not been proved is 
that $(f)$ can be added to the equivalences $(a)$ through $(e)$
in the general Abelian case. For that purpose assume that $(b)$ holds
and let $Z$ be stationary and independent of $(X, \xi)$ and $(Y, \eta)$.
Then $((Z,X), \xi)$ is the Palm version of $((Z,Y), \eta)$
and the equivalence of $(b)$ and $(e)$ yields~$(f)$.
Conversely, $(e)$ follows from $(f)$.
\end{proof}

\section{Measure preserving  allocations} 
\label{allocations}

For  a measurable map $\pi$ taking a random measure
$\xi$ to a location $\pi (\xi)$ in $G$, define the associated 
{\em $\xi$-allocation} $\tau$ by 
\begin{align*}
\tau(t) = \tau_{\xi}(t) = t\pi (t^{-1}\xi), 
\quad t \in G.
\end{align*}
Similarly, for a measurable map $\pi$ taking $(X,\xi)$ to a location 
$\pi (X, \xi)$ in $G$, define the associated 
 {\em $(X,\xi)$-allocation} $\tau$ by 
\begin{align*}
\tau(t) = \tau_{(X,\xi)}(t) = t\pi (t^{-1}(X,\xi)), 
\quad t \in G.
\end{align*}
The $\pi$ in the definition of reversibility above is defined for
simple point processes $\xi$ having a point at the origin.
If we define $\pi$ for simple 
point processes $\xi$ {\em not} having 
a point at the origin  by $\pi(\xi) = 0$ and  $\pi(X, \xi) = 0$, 
respectively, then 
$\pi$ is reversible if and only if the associated $\tau$ is
a bijection. 
The bijectivity of $\tau$ is further  equivalent to $\tau$ {\em preserving} 
the measure $\xi$, 
that is, for each fixed value of $\xi$ the image measure of $\xi$ under $\tau$ is 
$\xi$ itself:
$$\xi(\{s \in G : \tau(s) \in A\}) = \xi(A), \quad A \in \cG,$$ 
or in probabilistic notation,
\begin{align*}
\xi(\tau \in \cdot) = \xi.
\end{align*}
Preservation and bijectivity are, however, only equivalent
if we restrict to the simple point process case. 
Preservation (rather than reversibility/bijectivity) 
turns out to be the property that is essential for
going beyond simple point processes.

Say that $\pi$ is {\em preserving} if the associated $\tau$ preserves $\xi$.
In \cite{LaTho09} it is shown that the following analogue of 
\eqref{(5.1)},
\begin{align}\label{(6.1)}
T^{-1} (X, \xi) \overset{D}{=}  (X, \xi)
\text{ for all $T = \pi(\xi)$ where $\pi$ is preserving,} 
\end{align}
does {\em not} suffice
to characterize the Palm versions of a stationary random measures
with point masses of different positive sizes
since an allocation cannot split a positive point mass. 
Neither does \eqref{(6.1)} with $T = \pi(X, \xi)$ for the same reason.
One might therefore want to restrict attention to 
{\em diffuse}  random measures, that is,
random measure with no positive point masses.
It is not known yet whether \eqref{(6.1)} 
does suffice
to characterize Palm versions in the diffuse case.
However, this is true when $G = \R^d$ 
if stationary independent backgrounds are allowed.
The following result is from the forthcoming paper \cite{LaTho11}.

\begin{theorem} \label{6.1}
Let $\xi$ be a diffuse random measure on $\R^d$
having the origin in its support.
Then the following claims are equivalent:

$(a)$ 
the pair $(X, \xi)$ is mass-stationary,

$(b)$ 
for all stationary independent backgrounds $Z$,
\begin{align*}
T^{-1} ((Z,X), \xi) \overset{D}{=}  ((Z,X), \xi)
\text{ for all $T = \pi(Z,\xi)$ where $\pi$ is preserving,} 
\end{align*}

$(c)$ 
for all stationary independent backgrounds $Z$,
\begin{align*}
T^{-1} ((Z,X), \xi) \overset{D}{=}  ((Z,X), \xi)
\text{ for all $T = \pi((Z,X),\xi)$  where $\pi$ is preserving.} 
\end{align*}
\end{theorem}

\section{Cox and Bernoulli randomizations}
\label{cox}

Stationary independent backgrounds constitute a certain kind of
randomization. Another kind of randomization,
a Cox randomization, yields
a full characterization of mass-stationarity in the Abelian case as we now explain.

Consider a {\em Cox process} driven by $(X, \xi)$, that is,
an integer-valued point process which 
conditionally on $(X, \xi)$ is a
Poisson process with intensity  measure~$\xi$.
Intuitively, the Cox process can be thought of as representing the mass of  $\xi$
through a collection of points placed 
independently  at typical locations 
in the mass of $\xi$.
Thus if $(X, \xi)$ is  mass-stationary (if the origin is a typical
location for $X$ in the mass of $\xi$) 
and we add an extra point at the origin to 
the Cox process, then the points of that {\em modified} Cox process $N$
are {\em all} at typical locations in the mass of $\xi$.

It turns out that mass-stationarity reduces to
mass-stationarity with respect to this modified Cox process;
for proof see \cite{LaTho10}.

\begin{theorem} \label{7.1}
Let $\xi$ be a random measure on an Abelian $G$. 
Then  the following claims are equivalent:

$(a)$
the pair $(X, \xi)$ is  mass-stationary,

$(b)$
the pair $(X, N)$ is  mass-stationary,

$(c)$
the pair $((X,\xi), N)$ is  mass-stationary.
\end{theorem}

In the  diffuse case, 
the modified Cox process $N$ is a simple point process
and mass-stationarity reduces to point-stationarity
 by Theorem~\ref{5.2}:

\begin{corollary} \label{7.2}
 Let $\xi$ be a diffuse random measure on an Abelian $G$. 
Then the following claims are equivalent:

$(a)$ 
the pair $(X, \xi)$ is mass-stationary,

$(b)$ 
the pair $(X, N)$ is  point-stationary,

$(c)$
the pair $((X,\xi), N)$ is  point-stationary.
\end{corollary}
Due to this result the various reversible shifts  that are known for 
simple point processes can now be applied  
to diffuse random measures through the modified Cox process $N$.
 
Yet another kind of randomization, a Bernoulli randomization,
works in the discrete case. A {\em Bernoulli transport}
refers to a randomized allocation rule $\tau$ that allows staying 
at a location $s$ with a probability
$p(s)$ depending on $s^{-1}(X,\xi)$
and otherwise chooses another location according to a
(non-randomized) allocation rule. 
Call the associated $\pi$ {\em Bernoulli}.
This makes it possible to split discrete point-masses. 
The following result is from \cite{LaTho10}.
 
\begin{theorem} \label{7.3}
Let $\xi$ be a discrete random measure on an Abelian $G$. 
Then  $(X, \xi)$ is  mass-stationary
if and only if
\begin{align*}
T^{-1} (X, \xi) \overset{D}{=}  (X, \xi)
\end{align*}
for all $T = \pi(\xi)$ where $\pi$ is preserving and Bernoulli.
\end{theorem}
 
\section{Mass-stationarity through bounded invariant\\ kernels}
\label{transport}

We  conclude with a more analytical characterization 
of mass-stationarity. A kernel $K_{(X, \xi)}$ from $G$ to $G$ is
{\em preserving} if 
$$
\int K_{(X, \xi)}(s, A) \xi(ds) = \xi(A), \quad A \in \cG,
$$
and {\em invariant} if 
\begin{align*}
K_{(X, \xi)}(t, A) = K_{t^{-1}(X, \xi)}(0, t^{-1}A), 
\quad t \in G,\,
A \in \cG.
\end{align*}
Note that if $\tau$ is a preserving allocation then the kernel 
defined by
$$K_{(X, \xi)}(t, A) = 1_A(\tau(t))$$ is preserving and invariant.
It is also 
Markovian and 
therefore bounded.

In the Abelian case the following result is from \cite{LaTho09}.
For the general case, which~can~be handled as in Section 3.8 of [7],
we need the modular function $\Delta:G\rightarrow(0,\infty)$
of $G$ ($\Delta\equiv 1$ in the Abelian case).

\begin{theorem} \label{8.1}
The pair  $(X, \xi)$ is  mass-stationary
if and only if 
for all preserving invariant {\em bounded} kernels $K$ 
and all nonnegative measurable functions $f$,
\begin{align}\label{(8.1)}
\BE\left[\int f\big{(}s^{-1}(X,\xi)\big{)}\,\Delta(s^{-1})K_{(X, \xi)}(0,ds)\right]=\BE[f(X,\xi)].
\end{align}
\end{theorem}

If $G$ is Abelian and $K_{(X, \xi)}$ is  Markovian then \eqref{(8.1)}
means that 
\begin{align*}
T^{-1} (X, \xi) \overset{D}{=} (X, \xi)
\end{align*}
 where~$T$ has 
conditional distribution $K_{(X, \xi)}(0,\cdot)$ given  $(X,\xi)$.
It is not known yet whether  `bounded' in the theorem
can be replaced by `Markovian'.

Theorem~\ref{8.1} and Theorem~\ref{3.2} yield  
the following extension of Theorem~\ref{2.2}$(b)$ 
to the locally compact case.

\begin{theorem}\label{8.2}
The pair $(X, \lambda)$ is mass-stationary 
$($that is, the origin is a typical location for $X$$)$ 
if and only if $X$ is stationary.
\end{theorem}
\begin{proof}
Suppose $X$ is stationary. 
Then so is $(X, \lambda)$.  A stationary $(X, \lambda)$ 
is the Palm version of itself. Thus Theorem~\ref{3.2} yields the fact that 
$(X, \lambda)$ is mass-stationary.
Conversely, assume that $(X, \lambda)$ is mass-stationary.
Fix an arbitrary $t \in G$ and let $K_{(X, \lambda)}$ be the invariant kernel with
 $K_{(X, \lambda)}(0, A) = \Delta(t)1_A(t)$. This
 kernel is preserving and from   \eqref{8.1} we obtain that
$\BE\left[f\big{(}t^{-1}(X,\lambda)\big{)}\,\right]=\BE[f(X,\lambda)]$.
Since this holds for all nonnegative measurable $f$ it holds in particular for 
$f$ that are constant in the second argument and thus 
$X \overset{D}{=} Y$. Hence $X$ is stationary. 
\end{proof}

Theorem~\ref{8.2} shows that mass-stationarity is a generalization 
of the concept of stationarity.

\end{document}